\documentclass[12pt]{article}

\usepackage{amsmath}
\usepackage{amsthm}
\usepackage{amsfonts}
\usepackage{amssymb}
\usepackage{graphicx}
\usepackage{latexsym}

\advance \topmargin by -\headheight
\advance \topmargin by -\headsep
\evensidemargin \oddsidemargin
\newtheorem{theorem}{Theorem}[section]
\newtheorem{thm}{Theorem}[section]

\newtheorem{lem}[theorem]{Lemma}

\theoremstyle{definition}

\newtheorem{defn}[theorem]{Definition}
\newtheorem{example}[theorem]{Example}

\newtheorem{question}[theorem]{Question}
\theoremstyle{remark}

\numberwithin{equation}{section}

\newcommand\EE{\mathbb E}
\newcommand\HH{\mathbb H}
\newcommand\NN{\mathbb N}
\newcommand\RR{\mathbb R}
\newcommand\ZZ{\mathbb Z}

\newcommand\FF{\mathbb F}

\newcommand\cB{\mathcal{B}}

\newcommand\cC{\mathcal{C}}

\newcommand\cI{\mathcal{I}}

\newcommand\cL{\mathcal{L}}

\newcommand\cF{\mathcal{F}}

\newcommand\cR{\mathcal{R}}

\newcommand\Inn{\operatorname{Inn}}

\newcommand\hT{\widehat T}

\newcommand\norm[1]{\left\|#1\right\|}

\newcommand\abs[1]{\left|#1\right|}

\newcommand\set[1]{\left\{{#1}\right\}}
\def\SUM{{\textrm{SUM}}} %could do \mathbb S instead
\def\RATIO{{\textrm{RATIO}}} %could do \mathbb Q instead (for quotient)
\def\sS{{\textrm{S}}}

\def\cc{\curvearrowright}
\def\f{{\textrm{fin}}}

\begin{document}
%\title{The lattice point counting problem with
%uniform error estimates}
\title{A horospherical ratio ergodic theorem \\ for actions of free groups}

%    Information for first author
\author{Lewis Bowen\footnote{supported in part by NSF grant DMS-1000104 and BSF grant 2008274} ~and Amos Nevo\footnote{supported in part by ISF, and  BSF grant 2008274}}%\address{University of Hawaii}
%\email{lpbowen@math.hawaii.edu}
%\thanks{The first author was supported in part by NSF Grant}

%    Information for second author

%    Address of record for the research reported here
%\address{Institute for Advanced Study, Princeton}

%    Current address
%\curraddr{Department of Mathematics, Technion}
%\email{anevo@tx.technion.ac.il, nevo@math.ias.edu}
%    \thanks will become a 1st page footnote.
%\thanks{The second author was supported by the ISF}

%    General info
%\subjclass{Primary ; Secondary}

%\date{July 2008}

%\dedicatory{}

%\keywords{}
\maketitle

\begin{abstract}
We prove a ratio ergodic theorem for amenable equivalence relations satisfying a strong form of the Besicovich covering property. We then use this result to study general non-singular actions of non-abelian free groups and establish a ratio ergodic theorem for averages along horospheres.
\end{abstract}

\tableofcontents

\section{Introduction}

Consider a non-singular action of a countable group $G$ on a standard $\sigma$-finite measure space $(X, \cB,\eta)$, which we denote by $x \mapsto T^g x$. From the action on $X$, there is an induced isometric action on $L^\infty(X)$ also denoted by $T^g$ given by $T^gf(x)=f(T^{g^{-1}}x)$. This induces an isometric action on the  Banach pre-dual of $L^\infty(X,\eta)$, namely on $L^1(X,\eta)$ which is given by $\hT^g(f)=T^{g}f\cdot\frac{dg\eta}{d\eta}$. 

This set-up gives rise to a wide array of important and interesting examples of actions, where one would like to study the statistical properties of the distribution of the orbits of $G$ in $X$.  Very little is  known about this  problem for general groups, and let us begin by reviewing the main results in the case of actions of Abelian groups. 
%If $\mu$ is a probability measure on $G$, then $\mu$ induces an operator $T^\mu$ on $L^\infty(X)$ by $T^\mu f = \int T^g f~d\mu(g)$ and an operator $\hT^\mu$ on $L^1(X)$ by $\hT^\mu f = \int \hT^g f~d\mu(g)$. 

\subsection{The ratio ergodic theorem for commuting transformations}

For $\ZZ$-actions there is a generalization of Birkhoff's pointwise ergodic theorem due to Hopf [Ho37], later generalized by Hurewicz [Hu44] and extended to operators by Chacon-Ornstein [CO60]. Hopf's ratio ergodic theorem states that if $T:(X,\lambda) \to (X,\lambda)$ is non-singular and conservative, $u,v \in L^1(X,\lambda)$ and $\int v~d\lambda \ne 0$ then the ratios
$$\RATIO_n[u,v] := \frac{ \sum_{k=0}^n \hT^k u}{\sum_{k=0}^n \hT^k v}$$
converge pointwise almost everywhere as $n\to\infty$ to a function $r(u,v)$ on $X$ satisfying
\begin{itemize}
\item $r(u,v)\circ T = r(u,v)$, namely $r(u,v)$ is $T$-invariant, 
\item $\int f\cdot r(u,v)v~d\lambda = \int f \cdot u ~d\lambda$ for any $f\in L^\infty(X)$ such that $f\circ T=f$.
\end{itemize}
In particular, if $T$ is ergodic then $r(u,v)$ equals the constant $\frac{\int u~d\lambda}{\int v ~d\lambda}$ almost everywhere. In general,
$$r(u,v) = \EE_{\lambda_v}\left[\frac{u}{v}\Big|\cI\right]$$
is the conditional expectation of $\frac{u}{v}$ on the $\sigma$-algebra $\cI$ of $T$-invariant sets with respect to the measure $\lambda_v$ defined by $d\lambda_v=vd\lambda$.

This result has only recently been extended to $\ZZ^d$ actions by Feldman [Fe07] and by  Hochman [Ho10]. There are no known results of a similarly general nature for any non-amenable group; nor are there counterexamples. In particular the following basic problem is open.

\begin{question}\label{q:main}
Let $\FF=\langle a_1,\ldots, a_r\rangle$ be a rank $r \ge 2$ free group. Let $(T^g)_{g\in \FF}$ be a conservative action on a standard $\sigma$-finite measure space $(X, \cB,\lambda)$ by non-singular transformations. Let $(\hT^g)_{g\in \FF} $ be the induced isometric action on $L^1(X)$,  the Banach pre-dual of $L^\infty(X)$. For $g\in \FF$, let $|g|$ denote its word length with respect to the given generators. For $u,v \in L^1(X)$ with $\int v~d\lambda \ne 0$ let
$$\RATIO_n[u,v] := \frac{ \sum_{|g| \le n} \hT^g u}{\sum_{|g|\le  n} \hT^k v}.$$
Then does the sequence $\RATIO_{2n}[u,v]$ converge pointwise almost everywhere as $n\to\infty$ ? and if so, what is its limit ?
\end{question}
It is necessary to consider $\RATIO_{2n}[u,v]$ rather than $\RATIO_{n}[u,v]$ due to a certain well-known periodicity phenomenon which occurs if the action has an eigenfunction with eigenvalue $-1$.

The same question arises with $|g| \le n$ replaced with $|g|=n$ and 
with $L^1$ replaced with $L^p$ for  $p \in [1,+\infty]$.  It is also natural to consider averages other than the uniform averages over spheres,  for example random walk averages given by convolution powers.  One would also like to consider averages in groups other than free groups.  

\subsection{Ratio equidistribution of dense subgroups}
Part of the interest in this problem comes from the special case in which $(X,\cB,\lambda)$ is a locally compact group with Haar measure, $\phi: \Gamma \to X$ is a group homomorphism onto a dense subgroup, the action $(T^g)_{g\in \Gamma}$ is given by $T^g(x)=\phi(g)x$ and $u,v$ are compactly supported continuous functions. In this case, it is natural to ask whether the ratios $\RATIO_{n}[u,v]$ converge everywhere (instead of almost everywhere). 
%This problem of equidistribution of ratios for dense subgroups of lcsc groups will be referred to as the generalized Kazhdan equidistribution problem. 
%We will show in \S \ref{sec:equidist}  below that in this case almost everywhere convergence actually implies convergence everywhere. 

Arnol'd and Krylov were among the first to consider the problem of establishing equidistribution  for actions of free groups [AK63]. They established uniform convergence in the case when the group $X$ is $SO_3(\RR)$, but the same method applies whenever the group $X$ is compact.
We note that the case that the underlying invariant measure is finite the mean and subsequently the pointwise ergodic theorem for completely general actions of free groups were considered by several authors \cite{Gu68}\cite{Gr99}\cite{N1}\cite{NS}\cite{Bu}\cite{BN1}(see also the recent survey \cite{BK}). 
%Subsequently mean convergence was established for general probability measure preserving action by Guivarc'h \cite{Gu}.

 Kazhdan considered equidistribution of ratios in the case when $X$ is the isometry group of the Euclidean plane and the averages are random walk averages [Ka65]. His argument was corrected and the results extended by Guivarc'h [Gu76] (see also [Vo04]). An important advance in the case when $X$ equals the isometry group of Euclidean $n$-space was very recently obtained by Varju \cite{Va}. Breuillard has obtained positive results when $X$ is the Heisenberg group [Br05] and the averages are random walk averages. He has also obtained positive results when $X$ is any simply-connected nilpotent Lie group [Br10] and the averages are uniform over balls in the Cayley graph of the image group $\Phi(\Gamma)$. A survey of these results  can be found in  [Brxx]. 

\subsection{Some recent ratio ergodic theorems : lattice actions on homogeneous spaces}

Let us now describe some developments over the last decade, which pertain to actions of a lattice subgroup $\Gamma$ of a semisimple algebraic group $G$ on homogeneous spaces $G/H$ of $G$. In order to keep the exposition focused, let us concentrate on the case $G=SL_2(\RR)$, although all the results we will describe apply in much greater generality. 

  To gain some perspective on the subtlety of the problem of establishing ratio ergodic theorems and some of the difficulties that must be overcome, we consider the following three actions. 
Let $ \Gamma\subset SL_2(\RR)$ be any lattice subgroup, for example a lattice isomorphic to the free group $\FF_r$. 

\subsubsection{Lattice action on the boundary of the hyperbolic plane.}
Consider the homogeneous space $B=SL_2(\RR)/P=P^1(\RR)$ where $P$ is the group of upper triangular matrices. Fix a point $o\in G/K=\HH^2$, the $G$-invariant metric $d$, and the $K$-invariant probability measure $m$ on $B$. For any line $[v] \in B$, and any two continuous functions $\phi$ and $\psi$ on $B$, 
$$\lim_{T\to \infty} \frac{\sum_{d(\gamma o,o)\le T} \phi([v]\gamma)}{\sum_{d(\gamma o,o)\le T} \psi([v]\gamma)}=\frac{\int_{B}\phi([w])dm([w])}{\int_{\RR^2}\psi([w])dm([w])}
$$
This result is based on a precise asymptotics established for the expression 
$\sum_{d(\gamma o,o)\le T} \phi([v]\gamma)$ in \cite{G1}, from which the ratio equidistribution theorem is an immediate corollary.  The lattice action on the boundary $G/P$ was originally considered more generally by Gorodnik \cite{G1},  and developed further in Gorodnik-Maucourant \cite{gm} and  Gorodnik-Oh \cite{go}. 
 
\subsubsection{Lattice action on the Euclidean plane.}

Consider the homogeneous space $\RR^2\setminus\set{0}=SL_2(\RR)/N$, $N$ the upper triangular unipotent group. Fix any norm on $M_2(\RR)$.  Let $v\in \RR^2$ be any vector whose $\Gamma$-orbit is dense in $\RR^2$. Then given any two continuous compactly supported functions $\phi,\psi$ on $\RR^2$, 
$$\lim_{T\to \infty} \frac{\sum_{\norm{\gamma}\le T} \phi(v\gamma)}{\sum_{\norm{\gamma}\le T} \psi(v\gamma)}=\frac{\int_{\RR^2}\phi(w)\alpha_v(w)dw}{\int_{\RR^2}\psi(w)\alpha_v(w)dw}
$$
where $\alpha_v(w)$ is a positive continuous density on $\RR^2$ (and $dw$ denotes Lebesgue measure). Once again, this result is based on a precise asymptotics established for the expression 
$\sum_{\norm{\gamma}\le T} \phi(v\gamma)$  in Theorem 12.2 of [GW07], from which the ratio equidistribution theorem is an immediate  corollary. 

Let us note that the limiting density in the ratio theorem depends non-trivially on the starting point $v\in \RR^2$, and for any given $v$, depends non-trivially on the norm. This is shown explicitly in [GW07], which in turn generalizes previous work of Ledrappier  [LP99][LP01] and Ledrappier-Pollicott [LP03][LP05].

\subsubsection{Lattice action on the quadratic surface (de-Sitter space).}

Consider now the homogeneous space $X=SL_2(\RR)/A$, where $A$ is the diagonal group. It can be realized as the one-sheeted hyperboloid in $\RR^3$ given by the level set of a quadratic form of signature $(2,1)$ invariant under $SL_2(\RR)\cong SO^0(2,1)$. The quadratic surface in question is also called de-Sitter space.  Fix any norm on $M_2(\RR)$. 
Using  polar coordinate system on de-Sitter's space, namely 
$\mathbb{R}\times S^{1} \mapsto X: (r,\omega)\mapsto (\omega\cosh r, \sinh r)$, 
for every $\phi,\psi \in L^1(X)$ with compact support and
for almost every $v\in X$,
$$\lim_{T\to\infty} \frac{ \sum_{\norm{\gamma} < T}  \phi(v\gamma)}{ \sum_{\norm{\gamma }< T}  \psi(v\gamma)}  = \frac{ \int_{X} \phi(r,\omega)\, \cosh r\,dr\, d\omega}{ \int_{X} \psi(r,\omega)\, \cosh r\,dr\, d\omega}\,.$$
Note that the limiting density here is not the $G$-invariant measure on the quadratic surface, but it is in fact independent of the starting point and the norm in this case. 

Once again, this result is based on the precise asymptotics 
$\lim_{t\to\infty} \frac{1}{t} \sum_{\gamma\in\Gamma_t}
  \phi(v\gamma)= c_{2}(\Gamma)
 \int_{X} \phi(r,\omega)\, (\cosh r)\,dr\, d\omega$ 
(for some $c_2(\Gamma)>0$) established in \cite{GN2}. 

Further examples of precise asymptotics for actions of lattices in algebraic groups acting on homogeneous spaces with infinite invariant measure were established in \cite{GW}, and \cite{GN2}. These give rise to further ratio equidistribution theorems, or almost sure ratio ergodic theorems. A further discussion can be found in \cite{GN3}.

%Part of the interest in this problem comes from a paper of D. Kazhdan [Ka65] in which he proved the following:
%\begin{thm}
%Let $\FF=\langle a, b\rangle$ be a rank $2$ free group. Let $(T^g)_{g\in \FF}$ be an action of $\FF$ on the Euclidean plane $\EE^2$ by isometries. Assume that the action is minimal; i.e., for every $x\in \EE^2$, the orbit $\{T^gx:~g\in \FF\}$ is dense in $\EE^2$. Given a function $f$ on $\EE^2$ let $P[f]$ be the function 
%$$P[f](x):= \frac{1}{4} \left( f(ax) + f(bx) + f(a^{-1}x) + f(b^{-1}x)\right).$$
%Let $f,g$ be compactly supported continuous functions on the Euclidean plane with $g\ge 0$ nonzero. Let
%$$\RATIO_n[f,g]:= \frac{ P^n[f]}{P_n[g]}.$$
%Then $\RATIO_n[f,g]$ converges pointwise almost everywhere as $n\to\infty$ to $\frac{\int f(x)~dx}{\int g(x)~dx}$.
%\end{thm}
%This theorem has been extended by Y. Guivarch

\subsection{A dynamical approach to ratio theorems for free groups} 
This paper explores a new approach to Question \ref{q:main} related to a recent proof of the pointwise ergodic theorem for spherical averages in free groups [BN10], and its generalization in \cite{BN2}. To set the notation, let $\FF=\langle a_1,\ldots, a_r\rangle$ denote the free group of rank $r\ge 2$. Let $S=\cup_{i=1}^r \{a_i,a_i^{-1}\}$ be the free symmetric generating set. The {\em reduced form} of an element $g\in \FF$ is the unique expression of the form $g=s_1\cdots s_n$ with $s_i \in S$ and $s_{i+1}\ne s_i^{-1}$ for all $i$. Define $|g|:=n$, the length of the reduced form of $g$. Define a distance on $\FF$ by $d(g_1,g_2):=|g_1^{-1}g_2|$.

The boundary $\partial \FF$ is the subspace of all sequences $\xi=(s_1,s_2,\ldots) \in S^\NN$ such that $s_{i+1}\ne s_i^{-1}$ for all $i\ge 1$. It is naturally endowed with a Markov probability measure $\nu$ determined as follows. Let $(s_1,\ldots, s_n) \in S^n$ be an arbitrary sequence such that  $s_{i+1}\ne s_i^{-1}$ for all $i\ge 1$. Then 
$$\nu(\{ \xi \in \partial \FF:~ \xi_i=s_i~\forall 1\le i \le n\}) = (2r)^{-1}(2r-1)^{-(n-1)}.$$
There is a natural action of $\FF$ on $\partial \FF$ by 
$$(t_1\cdots t_n)(s_1,s_2,\ldots) = (t_1,\ldots,t_{n-k},s_{k+1},s_{k+2}, \ldots)$$ where $t_1,\ldots, t_n \in S$,  $g=t_1\cdots t_n \in \FF$ is in reduced form and $k$ is the largest number $\le n$ such that $s_i^{-1} = t_{n+1-i}$ for all $i\le k$.

The action preserves the measure class of $\nu$. For any $\xi \in \partial \FF$, the set 
$$H_\xi=\left\{g\in \FF:~ \frac{d\nu \circ g^{-1}}{d\nu}(\xi) = 1\right\}$$
is the {\em horosphere} centered at $\xi$ passing through the identity element $e$. It can alternatively be described as the set of all elements $g \in \FF$ such that for some $n=n(g)>0$, $(g^{-1}\xi)_i=\xi_i$ for all $i>n$ (i.e., $g^{-1}$ preserves the `tail' of $\xi$, from some point onwards).

 We let $\cR_0 \subset \partial \FF \times \partial \FF$ be the equivalence relation given by $\xi \sim_{\cR} \xi^\prime$ if and only if there is an $n>0$ such that $\xi_i =\xi_i^\prime$ for all $i>n$. In other words, $\xi \sim_{\cR_0} \xi^\prime$ if there exists $g\in H_\xi$ such that $g^{-1}\xi=\xi^\prime$. 
 
Now let $(X, \cB,\lambda)$ be a standard $\sigma$-finite measure space on which $\FF$ acts by non-singular transformations. Let $\cR_0(X) \subset X\times \partial \FF \times X \times \partial \FF$ be the equivalence relation under which $(x,\xi) $ is equivalent to $(y,\xi^\prime)$ if and only if there is a $g\in H_\xi$ such that $(y,\xi^\prime)=(g^{-1}x,g^{-1}\xi)$. % For  $f\in L^1(X\times \partial \FF)$, let $\EE[f|\cR_0(X)]$ be the conditional expectation of $f$ on the $\sigma$-algebra of $\cR_0(X)$-saturated sets (that is Borel sets which are unions of $\cR_0(X)$-equivalence classes).

Our main theorem is:
\begin{thm}\label{thm:main-0}
If $u,v \in L^1(X\times \partial \FF)$, $v > 0$ then the ratios
$$\RATIO_{2n}[u,v](x,\xi):= \frac{ \sum_{g^{-1} \in H_\xi, |g|\le 2n} \hT^gu(x,\xi)}{  \sum_{g^{-1} \in H_\xi, |g|\le 2n} \hT^gv(x,\xi)}$$
converge pointwise almost everywhere as $n\to \infty$ to a $\cR_0(X)$-invariant function $r(u,v)$ (which, if $\cR_0(X)$ is ergodic equals  $\frac{\int u~d\lambda \times \nu}{\int v ~d\lambda\times \nu}$).
\end{thm}
It is an open question whether the analogous statement with the condition $|g| \le 2n$ replaced by $|g|=2n$ holds true.

In the case when $(X,\cB,\lambda)$ is a probability space and the action is measure-preserving, we proved in [BN1] that the averages
$$ \frac{1}{\#\{g^{-1} \in H_\xi, |g|\le 2n\}} \sum_{g^{-1} \in H_\xi, |g|\le 2n} \hT^gu(x,\xi)$$
converge almost everywhere to $\EE_{\lambda\times\nu}[u|\cR_0(X)](x,\xi)$ which is the conditional expectation of $u$ on the $\sigma$-algebra of $\cR_0(X)$-saturated set. Moreover, if $u(x,\xi)=u(x)$ is a function of the first argument only then $\EE_{\lambda\times\nu}[u|\cR_0(X)](x,\xi)$ is equal to $\EE_\lambda[u|\FF^2]$, the conditional expectation of $u$ on the $\sigma$-algebra of $\FF^2$-invariant sets (where $\FF^2 < \FF$ is the index $2$ subgroup consisting of all $g\in \FF$ with $|g|$ even). In particular if $\FF^2$ acts ergodically on $(X,\cB,\lambda)$ then the above averages converge pointwise a.e. to the integral $\int u(x)~dx$. By averaging the above averages over the whole boundary we obtained pointwise ergodic theorems for uniform spherical averages in free groups [BN1].

In the general case, when $(X,\cB,\lambda)$ is $\sigma$-finite and the action is merely non-singular then there are examples (provided in \S \ref{sec:ergodic}) showing that $\EE_{\lambda\times\nu}[u|\cR_0(X)]$ does not necessarily equal $\EE_\lambda[u|\FF^2]$. Many interesting cases remain open. For example, it is not known whether $\EE_{\lambda\times\nu}[u|\cR_0(X)]=\EE_\lambda[u|\FF^2]$ in the special case when $X$ is the isometry group of Euclidean $n$-space and $\phi:\FF \to X$ is a homomorphism onto a dense subgroup and the action $(T^g)_{g\in \FF}$ is given by $T^gx =\phi(g)x$.

%A word about spherical averages for the free group....

\subsection{Organization of the paper}

To begin, in \S \ref{sec:er} we prove a ratio ergodic theorem for measured equivalence relations with respect to averages on subsets satisfying some extreme invariance properties. It seems likely that the hypotheses in this result can be relaxed. Next we show in \S 3 that the horospherical equivalence relation described above have natural subsets that satisfy the hypotheses of \S 2. This implies our main result for the free group. In \S \ref{sec:ergodic} we exhibit some examples pertaining to the issue of whether ergodicity of $\FF \cc (X,\cB,\lambda)$ implies ergodicity of the relation $\cR_0(X)$.

{\bf Acknowledgements}. L. B. would like to thank Mike Hochman for useful conversations. 

\section{A ratio ergodic theorem for amenable equivalence relations}\label{sec:er} %AMENABLE

%\subsection{Notation for equivalence relations}

Consider a standard $\sigma$-finite measure space $(B,\nu)$ with a Borel equivalence relation $\cR \subset B\times B$. We assume that $\cR$ is discrete and quasi-invariant (with respect to $\nu$). Discrete means that for each $b\in B$, the equivalence class $[b]$ of $b$ is at most countable. Quasi-invariant means that for any Borel $A \subset B$ with $\nu(A)=0$, its saturation $[A]:=\cup_{a \in A} [a]$ also has $\nu$-measure zero. Integrating the counting measures on the fibers of the {\em left projection} $(x,y) \in \cR \mapsto x$ gives the {\em left counting measure} $M$ on $\cR$ satisfying $dM(x,y)=d\nu(x)$. The {\em right counting measure} ${\check M}$ is defined by $d{\check M}(x,y)=dM(y,x)$. Because $\nu$ is quasi-invariant, these two measures are equivalent and $D(x,y):=\frac{dM}{d{\check M}}(x,y)$ is the {\em Radon-Nikodym cocycle}. Thus the cocycle identity $D(x,z)=D(x,y)D(y,z)$ is satisfied, for almost all $x,y,z\in B$..

 Suppose that $\cF=\{\cF_n\}_{n=1}^\infty$ is a sequence of measurable functions $\cF_n:B \to 2_\f^B$ (where $2_\f^B$ denotes the space of finite subsets of $B$) such that $\cF_n(b)$ is a subset of the $\cR$-equivalence class of $b$ for every $b$. We will need the following definitions.

\begin{enumerate}
\item Let $\Inn(\cR)$ denote the group of {\em inner automorphisms} of $\cR$. These are invertible Borel maps $\phi:B \to B$ with graph contained in $\cR$. A set $\Phi \subset \Inn(\cR)$ {\em generates $\cR$} if for all $(b_1,b_2) \in \cR$ there exists $\phi \in \langle \Phi \rangle$ such that $\phi(b_1)=b_2$ (where $\langle \Phi \rangle$ denotes the group generated by $\Phi$).

\item $\cF$ is {\em extremely asymptotically invariant} if $\lim_{n\to\infty} |\cF_n(b)| = +\infty$ for $\nu$-a.e. $b$ and there exists a countable generating set $\Phi \subset \Inn(\cR)$ such that for every $\phi \in \Phi$ and $b\in B$ there exists an $N=N(\phi,b)$ such that $n>N$ implies $\cF_n(b)=\phi(\cF_n(b))$.

\item $\cF$ satisfies {\em the extreme Besicovich property} if for every $(b,b') \in \cR$ and any $n\ge 0$, either $\cF_n(b) = \cF_n(b')$ or $\cF_n(b) \cap \cF_n(b') = \emptyset$. Moreover, we require that there exists an $N=N(b,b')$ such that $n\ge N$ implies $\cF_n(b)=\cF_n(b')$. 

\item $\cF$ is {\em anchored} if $b \in \cF_n(b)$ for every $n$.
 
\end{enumerate}

For $u,v \in L^p(B)$ with $\int v ~d\nu \ne 0$ consider the sums $\SUM^\cF_n[u] \in L^p(B)$ and ratios $\RATIO^\cF_n[u,v]$ defined by
\begin{eqnarray*}
\SUM^\cF_n[u](b)&:=& \sum_{b'\in\cF_n(b)} u(b')D(b',b),\\
\RATIO^\cF_n[u,v](b)&:=&\frac{\SUM^\cF_n[u](b)}{ \SUM^\cF_n[v](b)}.
\end{eqnarray*}
For $u \in L^1(B,\nu)$ let $\EE_\nu[u|\cR]$ denotes the conditional expectation of $u$ on the $\sigma$-algebra of $\cR$-invariant sets with respect to the measure $\nu$. 

The main result of this section is:
\begin{thm}\label{thm:pointwise}
If $\cF$ is extremely asymptotically invariant, anchored and satisfies the extreme Besicovich property then $\cF$ is a pointwise ratio ergodic sequence in $L^1$. Namely, for every $u,v\in L^1(B,\nu)$, with $v > 0$ on $B$, $\RATIO^\cF_n[u,v]$ converges pointwise almost everywhere as $n\to\infty$ to a function $r(u,v)$ on $B$ satisfying
\begin{itemize}
\item  $r(u,v)$ is an $\cR$-invariant function, 
\item $\int f\cdot r(u,v)v~d\nu = \int f \cdot u ~d\nu$ for any $f\in L^\infty(B)$ which is $\cR$-invariant.
\end{itemize}
In particular, if  $\cR$ is ergodic then $r(u,v)$ equals the constant $\frac{\int u~d\nu}{\int v ~d\nu}$ almost everywhere. In general,
$$r(u,v) = \EE_{\nu_v}\left[\frac{u}{v}\Big|\cR\right]$$
is the conditional expectation of $\frac{u}{v}$ on the $\sigma$-algebra of $\cR$-invariant sets with respect to the measure $\nu_v$ defined by $d\nu_v=vd\nu$.
 \end{thm}

We would like to point out that the hypotheses are certainly not necessary. For example, the classical ratio ergodic theorem for averages along intervals of $\ZZ$ could be phrased in the language of equivalence relations. But expanding intervals in $\ZZ$ are neither extremely asymptotically invariant nor do they satisfy the extreme Besicovich property. It is an open problem to determine general hypotheses on $\cF$ guaranteeing that it is a ratio ergodic sequence.

The proof of Theorem \ref{thm:pointwise} follows a classical recipe: we first prove that it holds true for fixed $v$ and for $u$ in a naturally defined dense subset of $L^1$. Then with the aid of a maximal inequality, we show that it holds for all $u$ in $L^1$. To be precise:

\begin{thm}[Dense set of good functions]\label{thm:dense}
If $\cF$ is extremely asymptotically invariant then, given $ v\in L^1(B)$ with $v >0$ on $B$, there exists a norm dense subspace $\cL_v \subset L^1(B)$ such that for all $u \in \cL_v$, $\RATIO^\cF_n[u,v]$ converges pointwise almost everywhere. 
\end{thm}

Let $M^\cF_n[u,v] := \sup_n |\RATIO^\cF_n[u,v]|$. We will prove the following weak-type $(1,1)$ maximal inequality. 

\begin{thm}[$L^1$ maximal inequality]\label{thm:maximal}
Suppose that $\cF$ is anchored and satisfies the extreme Besicovich property. Then for any $u,v \in L^1(B)$ with $v \ge 0$ (and $v$ not identically zero) and any $\epsilon>0$,
$$\nu_v\Big( \big\{b \in B:~ M^\cF[u,v](b)\ge \epsilon \big\}\Big) \le  \frac{1}{\epsilon}\int_{\set{M^\cF[u,v](b)>\epsilon}}u(b)~d\nu(b) \le \frac{\norm{u}_{L^1(\nu)}}{\epsilon}$$
where $\nu_v$ is the measure $d\nu_v=vd\nu$.

\end{thm}

\subsection{A dense set of good functions}

\begin{lem}\label{u_phi}
For any $f \in L^1(\nu)$ and $\phi \in \Inn(\cR)$,
$$\int f(b)D(\phi b, b)~d\nu(b)=\int f(\phi^{-1}b)~d\nu(b).$$
\end{lem}

\begin{proof}
Recall that $D(\phi b, b) ~d{\check M}(\phi b, b) = dM(\phi b, b)$. Therefore
\begin{eqnarray*}
\int f(b)D(\phi b, b)~d\nu(b) = \int f(b)D(\phi b, b) ~d{\check M}(\phi b, b) = \int f(b) ~dM(\phi b, b) = \int f(\phi^{-1} b) ~d\nu(b).
\end{eqnarray*}
\end{proof}

\begin{lem}\label{lem:u_phi}
Let $\cF$ be extremely asymptotically invariant and $\Phi \subset \Inn(\cR)$ be a countable generating set witnessing its extreme asymptotic invariance. For $\phi \in \Phi$ and $u \in L^1(B)$, define $u_\phi\in L^1(B)$ by
$$u_\phi(b) := u(b) - u(\phi(b)) D(\phi(b),b).$$
Then  $\SUM^\cF_n[u_\phi]$ converges pointwise almost everywhere to $0=\EE_\nu[u_\phi|\cR]$. Moreover the span of $\{u_\phi:~u\in L^1(B), \phi \in \Phi\}$ is norm dense in $L^1_0(B)$ (= the set of $w \in L^1(B)$ with $\EE_\nu[w|\cR]=0$). 
\end{lem}

\begin{proof}
Note
$$\SUM^\cF_n[u_\phi](b)  = \sum_{b' \in \cF_n(b)} u(b') D(b',b) - u(\phi(b'))D(\phi(b'),b) = 0$$
for all $n>N(\phi,b)$ since $\phi(\cF_n(b)) = \cF_n(b)$. In particular $\SUM^\cF_n[u_\phi]$ converges pointwise almost everywhere to $0$.
%\sum_{b' \in \cF_n(b) \Delta \phi(\cF_n(b))} \epsilon(b,b')u(b') D(b,b')$$
%where $\epsilon(b,b')=1$ if $b' \in \cF_n(b) - \phi(\cF_n(b))$ and equals $-1$ otherwise.
%Without loss of generality, we may assume that $(B,\nu)$ is a probability space and $v$ is the constant function $1$. So
%$$|\RATIO^\cF_n[u_\phi,1]| \le ||u||_\infty \frac{ \sum_{b' \in \cF_n(b) \Delta \phi(\cF_n(b))} D(b,b')}{ \sum_{b' \in \cF_n(b)} D(b,b')}.$$
%**At this point, I think the key tool is something around Chacon-Ornstein.
By Lemma \ref{u_phi}, $\EE_\nu[u_\phi|\cR]=0$. To see that the span of the set of all functions of the above form is norm dense in $L^1_0(B)$, let $L^\infty_0(B)$ be the Banach dual of $L^1_0(B)$. Suppose $f \in L^\infty_0(B)$, $\int f u_\phi~d\nu=0$ for all $u \in L^1_0(B)$ and $\phi \in \Phi$. Then
$$\int f(b) u(b)~d\nu(b) = \int f(b) u(\phi(b))D(\phi(b),b)~d\nu(b) = \int f(\phi^{-1}b) u(b)~d\nu(b)$$
by the previous lemma. Since this equality holds for all $u \in L^1_0(B)$, it follows that $f = f\circ \phi^{-1}$ a.e. for every $\phi \in \Phi$. Since $\Phi$ is countable and generating, this implies that $f$ is $\cR$-invariant. So $f=\EE_\nu[f|\cR]=0$ a.e. This proves that the weak closure of the span of the collection of functions of the form $u_\phi$ (with $u\in L^1(B)$ and $\phi \in \Phi$) is all of $L^1_0(B)$. Because the weak closure of a subspace equals its norm closure, this proves the lemma.
\end{proof}

We will now use these facts in order to construct, given $v\in L^1(B)$, a norm-dense subset $\cL_v\subset L^1(B)$, such that for $u\in \cL_v$, $\frac{\SUM^\cF_n[u]}{\SUM^\cF_n[v]}$ converges almost everywhere.

\begin{proof}[Proof of Theorem \ref{thm:dense}]
Consider the linear subspace $\cL_v$ of $L^1(\nu)$ spanned by all functions of the form 
$wv$, for all bounded $L^1$ functions $w$ which are $\cR$-invariant, together with all the functions of the form $u_\phi$, as $\phi$ ranges over the set $\Phi$ of inner automorphisms defined above, and $u$ over $\left(L^1\cap L^\infty\right)(\nu)$. Thus 

$$\cL_v=\set{wv+u_\phi\,\,;\,\, w\in L^\infty(\nu)^\cR\,,\, u\in L^1\cap L^\infty( \nu), ~\phi \in \Phi} 
$$

Since $w$ is $\cR$-invariant, clearly $\RATIO^\cF_n[wv,v]=w$.   Lemma \ref{lem:u_phi} implies that  for every function $y\in \cL_v$ the ratios $\RATIO^\cF_n[y,v]$ converge almost everywhere. To show that $\cL_v$ is norm dense in $L^1( \nu)$, assume that $k\in L^\infty( \nu)$ satisfies $\int_{B} ky ~d \nu=0$ for all $y\in \cL_v$, and we will show that $k=0$. But under the foregoing condition  we have in particular that 
$$\int k(b) u(b)~d\nu(b) = \int k(b) u(\phi(b))D(\phi(b),b)~d\nu(b) = \int k(\phi^{-1}b) u(b)~d\nu(b).$$
 for every $u\in L^1( \nu)$ and $\phi \in \Phi$ by Lemma \ref{u_phi}. So $k$ must be invariant under $\Phi$. As we noted in the proof of Lemma \ref{lem:u_phi} above, it follows that $k$ is an $\cR$-invariant function. Therefore by definition $kv\in \cL_v$, and hence 
$\int_{B} k \cdot kv  ~d\nu=0$, so that $k$ vanishes on the support of $v$ which is all of $B$. Thus $\cL_v$ is weakly dense and therefore norm dense in $L^1(\nu)$. This concludes the proof of Theorem \ref{thm:dense}. 
\end{proof}

\subsection{Identifying the limit}
 We proceed to identify the limit of $ \RATIO^\cF_n[u,v]$ as the $\cR$-invariant function given by 
the conditional expectation  $ \EE_{\nu_v}\left[\frac{u}{v}\Big|\cR\right] $, where $d\nu_v=vd\nu$.  

\begin{lem}\label{lem:identify}
Let $\cL_v$ be the subspace defined in the proof of Theorem \ref{thm:dense}. Then for every $y \in \cL_v$, 
$$\lim_{n\to \infty}\RATIO^\cF_n[y,v]=  \EE_{\nu_v}\left[\frac{y}{v}\Big|\cR\right].$$
Moreover, for any $z \in L^1(\nu)$ and $\epsilon>0$ there exists $y \in \cL_v$ with $\EE_{\nu_v}\left[\frac{y}{v}\Big|\cR\right] = \EE_{\nu_v}\left[\frac{z}{v}\Big|\cR\right]$ and $\norm{u-y}_{L^1( \nu)}< \epsilon$.
\end{lem}

\begin{proof}
First, we already saw in the proof of Theorem \ref{thm:dense} above that if $wv+u_\phi=y\in \cL_v$ with $w$ an $\cR$-invariant function, then $\lim_{n\to \infty}\RATIO^\cF_n[y,v]=w$. If $y$ has another representation $y^\prime=w^\prime v+u^\prime_{\phi^\prime} $, then it follows that $w=w^\prime$, so that $w$ is uniquely determined by $y$. 

Thus we can consider the well-defined map $y=wv+u_\phi\mapsto wv=\Psi_v(y)$. We claim that 
 $\Psi_v : \cL_v\to L^1(\nu)$ is a contraction when taking 
 the $L^1( \nu)$-norm on both sides. Indeed, for any bounded $\cR$-invariant function $k$, we have 
 $$\int_{B}k y~d\nu=\int_{B}\left( kwv+ku_\phi\right)~ d\nu=\int_{B}kwv ~d \nu$$
Taking $k=\text{sign}(w)$, $k$ is obviously $\cR$-invariant since $w$ is, and thus 
$$\norm{\Psi_v(y)}_{L^1(\nu)}
%=\norm{W}_{L^1((\lambda\times \nu)_V)}
=\norm{wv}_{L^1( \nu)}=
\int_{B}\abs{w}v~d \nu=$$
$$=\int_B \text{ sign}(w) w v ~d\nu=\int_{B}\text{sign}(w)y ~d \nu\le \norm{y}_{L^1( \nu)}\,\,.$$
It follows that $\Psi_v : \cL_v\to L^1( \nu)$ can be extended to a linear operator of norm bounded by $1$ from the closure of $\cL_v$, namely $L^1( \nu)$, to $L^1( \nu)$. 

 Note that in actuality, looking first at functions in  $\cL_v$ and then at arbitrary functions in  $L^1( \nu)$, we see that the range of $\Psi_v$ is contained  in the closed subspace of $L^1(\nu)$ consisting of functions of the form $wv$, where $w$ is an $\cR$-invariant function. This space is naturally identified with the Banach space of $\cR$-invariant functions in $L^1(\nu_v)$. 

Clearly, the foregoing shows that for every $y\in \cL_v$, $\Psi_v(y)=wv$ has the property that for every bounded $\cR$-invariant function $k$, $$\int_{B} k\Psi_v(y)~d \nu=\int_{B} ky ~d\nu\,.$$
By the norm density of $\cL_v$ and by continuity of $\Psi_v(y)$, it follows that the last equation is valid also for $\Psi_v(z)$ for every $z\in L^1(\nu)$. 

We now claim that for every $y\in \cL_v$, 
$$\lim_{n\to \infty} 
\RATIO^\cF_n[y,v]= w =\frac{\Psi_v(y)}{v}
=  \EE_{\nu_v}\left[\frac{y}{v}\Big|\cR\right]. $$
Indeed the convergence of the ratios to $w$ was established above, and the fact that $w$ is the stated conditional expectation is equivalent to $$\int_B k w \cdot v~d\nu=\int_B k \frac{y}{v} \cdot v~d\nu$$
for every $k\in L^\infty(\nu)^\cR$, which was just verified. Because the conditional expectation is continuous and $\cL_v$ is norm-dense, it follows that
$$\frac{\Psi_v(z)}{v}
=  \EE_{\nu_v}\left[\frac{z}{v}\Big|\cR\right]$$
for every $z \in L^1(\nu)$. 

Now let $z \in L^1(\nu)$ and $\epsilon>0$. Because $\Psi_v$ has norm at most $1$ and $\cL_v$ is norm-dense, there is a $y \in \cL_v$ such that $\Psi_v(y)=\Psi_v(z)$ and $\norm{y-z}_{L^1( \nu)}< \epsilon$. The condition  $\Psi_v(y)=\Psi_v(z)$ implies $\EE_{\nu_v}\left[\frac{y}{v}\Big|\cR\right] = \EE_{\nu_v}\left[\frac{z}{v}\Big|\cR\right]$ as required.

\end{proof}

%\bar{\cE}_V\left(\frac{Y}{V}\right)$.
We will now show that the property that 
$\lim_{n\to \infty} 
\RATIO^\cF_n[u,v]=  \EE_{\nu_v}\left[\frac{u}{v}\Big|\cR\right] $ extends to all $u\in L^1( \nu)$, by an argument that employs the weak-type $(1,1)$ ratio maximal inequality. 

\subsection{Applying the ratio maximal inequality}

\begin{proof}[Proof of Theorem \ref{thm:pointwise} assuming Theorem \ref{thm:maximal}]

Let  $u$ be any function in $L^1( \nu)$ and $\epsilon > 0$. By Lemma \ref{lem:identify} there exists a function $y\in \cL_v$ which approximates $u$, namely $\norm{u-y}_{L^1( \nu)}< \epsilon^2$, and in addition $\EE_{\nu_v}\left[\frac{y}{v}\Big|\cR\right] = \EE_{\nu_v}\left[\frac{u}{v}\Big|\cR\right]$. 

Hence, applying the fact that $\lim_{n\to \infty} \RATIO^\cF_n[y,v]=\EE_{\nu_v}\left[\frac{y}{v}\Big|\cR\right] $ almost surely, we have 
\begin{eqnarray*}
&&\limsup_{n\to \infty}\abs{\RATIO^\cF_n(u,v)-\EE_{\nu_v}\left[\frac{u}{v}\Big|\cR\right]}\\
 &\le& \limsup_{n\to \infty}\abs{ \RATIO^\cF_n(u-y,v)} + \limsup_{n\to \infty}\abs{\RATIO^\cF_n(y,v)-\EE_{\nu_v}\left[\frac{u}{v}\Big|\cR\right]} \\
&=&\limsup_{n\to \infty}\abs{ \RATIO^\cF_n(u-y,v)}.
\end{eqnarray*}
%$$\le \sup_{n\in \NN} \abs{ \RATIO^\cF_n(u-y,v)}$$
By Theorem \ref{thm:maximal},
$$\nu_v\set{\limsup_{n\to \infty}\abs{ \RATIO^\cF_n(u-y,v)}> \epsilon}\le \frac{1}{\epsilon}\norm{u-y}_{L^1( \nu)} \le \epsilon.$$
Since this inequality is valid for all $\epsilon > 0$, we have 
$$\limsup_{n\to \infty} \abs{ \RATIO^\cF_n(u,v)-\EE_{\nu_v}\left[\frac{u}{v}\Big|\cR\right]}=0$$ 
almost everywhere and this completes the proof of Theorem \ref{thm:pointwise}    subject to the proof of Theorem \ref{thm:maximal}, to which we now turn.  

\end{proof}

\subsection{The weak-type  ratio maximal inequality in $L^1$}

\begin{proof}[Proof of Theorem \ref{thm:maximal}]
Without loss of generality, we may assume that $u$ is nonnegative. For $T>0$, define $M^\cF_T[u,v]:B \to \RR$ by 
$$M^\cF_T[u,v](b) := \sup_{T \ge n \ge 0} \RATIO^\cF_n[u,v](b).$$
Let $\epsilon>0$. Let $D$ be the set of all $b \in B$ such that $M^\cF[u,v](b)>\epsilon$ and $D_T$ be the set of all $b \in B$ such that $M^\cF_T[u,v](b)>\epsilon$. Since $\{D_T\}_{T>0}$ is an increasing sequence with $D = \cup_{T>0} D_T$, it suffices to prove that $\nu_v(D_T) \le  \epsilon^{-1}\int_{D_T} u~d\nu$ for each $T>0$. 

Fix $T>0$ and  let $\rho:D_T \to \RR$ be defined by $\rho(b)=n$ if $n$ is the largest number such that $\SUM^\cF_{n}[u,v](b) > \epsilon$ and $n\le T$. Let $\cC$ be the collection of all sets of the form $\cF_{\rho(b)}(b)$ for $b \in D_T$. Let $E$ be the union of all sets in $\cC$. Because $\cF$ is anchored, it follows that $D_T \subset E$. Now the extreme Besicovich property of $\cF$ implies $D_T=E$ and the sets in $\cC$ are pairwise disjoint.

 Define $F_u:B \times B \to \RR$ by
\begin{itemize}
\item  $F_u\big( b,b' \big) = u(b)|\cF_{\rho(b)}(b) |^{-1}$ if $b\in E$ and $b' \in \cF_{\rho(b)}(b) $. %Here $n$ is the number of spheres of $\cC'$ that contain both $b$ and $b'$.
\item $F_u\big( b,b' \big) = 0$ otherwise.
\end{itemize}
Define $F_v$ similarly with $v$ in place of $u$. 

For any $b,b'$ either $\cF_{\rho(b)}(b)  = \cF_{\rho(b')}(b') $ or $\cF_{\rho(b)}(b)  \cap \cF_{\rho(b')}(b') =\emptyset$. Therefore,
\begin{eqnarray*}
\int_E u(b) ~d\nu(b) &=& \int_{B}\sum_{b'} F_u\big( b,b' \big) ~d\nu(b)= \int_{\cR} F_u(b,b')~ dM(b,b')\\
&=&\int_{\cR} F_u(b,b') D(b,b') ~d{\check M}(b,b') =  \int_{B} \sum_{b} F_u\big( b,b' \big)D(b,b')~ d\nu(b')\\
&=& \int_E \frac{1}{| \cF_{\rho(b')}(b') |} \sum_{b\in\cF_{\rho(b')}(b') } u(b) D(b,b')~d\nu(b') =\int_E \frac{1}{| \cF_{\rho(b')}(b') |} \SUM^\cF_{\rho(b')}[u](b') ~d\nu(b').
\end{eqnarray*}
Similarly,
\begin{eqnarray*}
\int_E v(b) ~d\nu(b) &=&\int_E \frac{1}{|\cF_{\rho(b')}(b')|} \SUM^\cF_{\rho(b')}[v](b') ~d\nu(b').
\end{eqnarray*}
By definition of $D_T$,
$$\int_{D_T} \frac{1}{| \cF_{\rho(b')}(b') |} \SUM^\cF_{\rho(b')}[u](b') ~d\nu(b') \ge \epsilon \int_{D_T} \frac{1}{|\cF_{\rho(b')}(b')|} \SUM^\cF_{\rho(b')}[v](b') ~d\nu(b').$$
Since $E=D_T$ we now have
$$\int_{D_T} u(b) ~d\nu(b)  \ge \epsilon \int_{D_T} v(b) ~d\nu(b)$$
which implies
$$\nu_v(\set{M^\cF_T[u,v](b)>\epsilon})=\nu_v(D_T)\le  \frac{1}{\epsilon}\int_{\set{M^\cF_T[u,v](b)>\epsilon}}u(b)~d\nu(b) \le \frac{\norm{u}_{L^1(\nu)}}{\epsilon}.
$$
The theorem follows by letting $T \to \infty$.
\end{proof}

  \subsection{A ratio maximal inequality in $L^p$}

Let us note that the weak type $(1,1)$ maximal inequality can be used to derive a strong type $L^p$ maximal inequality, as follows. 
\begin{thm}\label{maxLp}
 Let $u,v\in L^1(B,\nu)$ with $\EE_\nu[v|\cR] > 0$ (where $\EE_\nu[\cdot|\cR]$ denotes conditional expectation with respect to the sigma-algebra of $\cR$-invariant Borel sets). Let
$$M^\cF[u,v]:=\sup_{n} \abs{\RATIO^\cF_{n}[u,v]}.$$
 If $\cF$ is anchored and satisfies the extreme Besicovich property then the strong type $L^p$ ratio maximal inequality holds, namely, assuming $\frac{u}{v}\in L^p( \nu_v)$
$$\norm{ M^\cF[u,v]}_{L^p( \nu_v)}^p\le 
\frac{p}{p-1}\norm{\frac{u}{v}}_{L^p( \nu_v)}^p$$ 
or more explicitly 
\begin{eqnarray*}
\int_{B} M^\cF[u,v]^p(b) v(b)~d\nu \le \frac{p}{p-1}\int_{B}\left(\frac{u(b)}{v(b)}\right)^p  v(b)~d\nu< \infty.
\end{eqnarray*}
\end{thm}

\begin{proof}
Let us first recall the following basic fact (see \cite{Ga70}). Suppose two non-negative measurable functions $\Phi$ and $\Psi $ on a standard $\sigma$-finite measure space $(Y,\eta)$ satisfy
\begin{enumerate}
\item $\Psi \in L^p(Y, \eta)$, for some $1 < p < \infty$, 
\item $\eta\set{y\,;\, \Phi(y)> \epsilon}< \infty$ for all $\epsilon > 0$, 
\item $\eta\set {y\,;\, \Phi(y) > \epsilon}\le \frac{1}{\epsilon} \int_{\set{y\,;\, \Phi(y) > \epsilon} }\Psi(y)~d\eta\,.$
\end{enumerate} 
Then $\Phi$ is in $L^p(Y,\eta)$, and $\norm{\Phi}_{L^p(Y,\eta)}\le \frac{p}{p-1}\norm{\Psi}_{L^p(Y,\eta)}\,.$

Fixing $v\ge 0$, let us consider the measure space $(Y,\eta)=(B, \nu_v)$, and define $\Phi=M^\cF[u,v]$, and $\Psi=\frac{u}{v}$. Then if both $u$ and $v$ are in $L^1( \nu)$, we have by Theorem \ref{thm:maximal}
$$\eta\set{y\,;\, \Phi(y)> \epsilon}=\nu_v\set{M^\cF[u,v](b) > \epsilon}\le \frac{1}{\epsilon}
\int_{\set{M^\cF[u,v](b)>\epsilon}}\frac{u(b)}{v(b)} v(b)~d\nu=\frac{1}{\epsilon} \int_{\set{y\,;\, \Phi(y) > \epsilon} }\Psi(y)~d\eta
$$
so that condition 2 and 3 above are satisfied. 
Assuming in addition that $\frac{u}{v}\in L^p( \nu_v)$, condition 1 above is satisfied as well, and we can conclude that 
$$\norm{M^\cF[u,v]}_{L^p( \nu_v)}\le \frac{p}{p-1}\norm{\frac{u}{v}}_{L^p( \nu_v)}$$
and the proof of the $L^p$ maximal inequality is complete. \end{proof}

\subsection{Cocycles}\label{sec:cocycle}

Given an equivalence relation $\cR$ and a sequence $\cF$ of subset functions with extreme properties, one can use cocycles defined on $\cR$ in order to push $\cF$ forward to a subset function with extreme properties defined on an extension of $\cR$. This fact will be  applied  below to actions of free groups. In general, suppose $(X,\lambda)$ is a standard measure space, $\Gamma$ is a group acting by measure-class-preserving transformations on $X$  and $\alpha:\cR \to \Gamma$ is a Borel cocycle. This means that $\alpha$ satisfies the cocycle equation:
$$\alpha(b_1,b_2)\alpha(b_2,b_3)=\alpha(b_1,b_3)$$
for a.e. $b_1,b_2,b_3 \in B$. Let $\cR^\alpha=\cR^\alpha(X)$ be the equivalence relation on $B\times X$ given by $(b,x)\sim_{\cR^\alpha}(b',x')$ if $(b,b') \in \cR$ and $\alpha(b',b)x=x'$. 
\begin{lem}\label{lem:cocycle}
Suppose $\cF=\{\cF_n\}_{n=1}^\infty$ is a sequence of subset functions for $\cR$. Let $\cF^\alpha=\{\cF^\alpha_n\}_{n=1}^\infty$ be defined by 
$$\cF^\alpha_n(b,x)=\{(b',x')\in B\times X:~ b' \in \cF_n(b)~ x' = \alpha(b',b)x\}.$$
If $\cF$ is anchored then $\cF^\alpha$ is anchored. If $\cF$ is extremely asymptotically invariant then $\cF^\alpha$ is extremely asymptotically invariant and if $\cF$ has the extreme Besicovich property then $\cF^\alpha$ has the extreme Besicovich property.
\end{lem}
\begin{proof}
We will show that if $\cF$ is extremely asymptotically invariant then $\cF^\alpha$ is also extremely asymptotically invariant. The other claims are similar. Let $\Phi \subset \Inn(\cR)$ be a countable generating set witnessing the extreme asymptotic invariance. For each $\phi \in \Phi$, define $\phi_\alpha \in \Inn(\cR^\alpha)$ by $\phi_\alpha(b,x)=(\phi(b), \alpha(\phi(b),b)x)$. Then $\Phi_\alpha:=\{\phi_\alpha:~\phi \in \Phi\} \subset \Inn(\cR^\alpha)$ is a countable generating set witnessing the extreme asymptotic invariance of $\cF^\alpha$.
\end{proof}

\section{A horospherical ratio ergodic theorem for the free group}
Let $\FF=\langle a_1,\dots ,a_r \rangle$ be the free group of rank $r\ge 2$. The {\em reduced form} of an element $g\in \FF$ is an expression of the form $g=s_1\cdots s_n$ with $s_i \in \sS$ and $s_{i+1}\ne s_i^{-1}$ for all $i$. The reduced form of $g$ is unique. Define $|g|:=n$, the length of the reduced form of $g$. Define a distance on $\FF$ by $d(g_1,g_2):=|g_1^{-1}g_2|$.

\subsection{The boundary}

The boundary of $\FF$ is the set of all sequences $\xi=(\xi_1,\xi_2,\ldots) \in \sS^\NN$ such that $\xi_{i+1} \ne \xi_i^{-1}$ for all $i\ge 1$. We denote it be $\partial \FF$. 

We define a probability measure $\nu$ on $\partial \FF$ as follows. For every finite sequence $t_1,\ldots, t_n$ with $t_{i+1} \ne t_i^{-1}$ for $1\le i <n$, let
$$\nu\Big(\big\{ (\xi_1,\xi_2,\ldots) \in \partial \FF :~ \xi_i=t_i ~\forall 1\le i \le n\big\}\Big) :=(2r-1)^{-n+1}(2r)^{-1}.$$
By the Carath\'eodory extension theorem, this uniquely extends to a Borel probability measure $\nu$ on $\partial \FF$.

\subsection{Horospheres}\label{sec:horospheres}
There is a natural action of $\FF$ on $\partial \FF$ by 
$$(t_1\cdots t_n)\xi = (t_1,\ldots,t_{n-k},\xi_{k+1},\xi_{k+2}, \ldots)$$ where $t_1,\ldots, t_n \in \sS$,  $t_1\cdots t_n$ is in reduced form and $k$ is the largest number $\le n$ such that $\xi_i^{-1} = t_{n+1-i}$ for all $i\le k$.  Observe that if $g=t_1 \cdots t_n$ then the Radon-Nikodym derivative satisfies
$$\frac{d\nu \circ g}{d\nu}(\xi) = (2r-1)^{2k-n}.$$
Let $\cR_0$ be the tail equivalence relation on $\partial \FF$ given by $(\xi,\eta)\in\cR_0$ if there exists $N\ge 1$ such that $n \ge N \Rightarrow \xi_n=\eta_n$. Given $\xi\in \cR_0$, the set $H_\xi=\{g \in \FF:~(\xi,g^{-1}\xi) \in \cR_0\}$ is the {\em horosphere} centered at $\xi$ passing through the identity $e$. Note that $\frac{d\nu \circ g^{-1}}{d\nu}(\xi)=1$ if $g \in H_\xi$. So $\nu$ is $\cR_0$-invariant.

Let $\cB=\{\cB_n\}$ be the sequence of subset functions given by 
$$\cB_n(\xi)=\{ \eta \in \partial \FF:~ \eta_k = \xi_k~\forall k > n\}.$$
In other words, $\cB_n(\xi)$ is the set of all $g\xi$ where $g$ ranges over the intersection of the horosphere $H_\xi$ with the ball of radius $2n$ in $\FF$.

In the next section we show:
\begin{thm}\label{thm:extreme1}
The sequence $\cB$ is anchored, extremely asymptotically invariant and satisfies the extreme Besicovich property.
\end{thm}

\subsection{Extreme properties of $\cB$}
\begin{defn}
 We will say that a map $\phi:\partial \FF \to \partial \FF$ has {\em order $\le n$} if there is a bijection $\phi':S^n \to S^n$ such that
 \begin{enumerate}
 \item for any $\xi \in \partial \FF$, $\phi(\xi_1,\ldots, \xi_n) = (\xi'_1,\ldots, \xi'_n, \xi_{n+1},\xi_{n+2},\ldots)$ where $\phi'(\xi_1,\ldots, \xi_n) = (\xi'_1,\ldots, \xi'_n)$;
 \item in the above $\xi'_n=\xi_n$ and, since $\phi(\xi) \in \partial \FF$, $\xi'_{i+1} \ne (\xi'_i)^{-1}$ for any $i$.
 \end{enumerate}
 \end{defn}

\begin{lem}
The set of finite order inner automorphisms of $\cR_0$ generates $\cR_0$.
\end{lem}

\begin{proof}
If $(\xi,\eta) \in \cR_0$ then there exists an $N$ such that $n\ge N$ implies $\xi_n=\eta_n$. Let $\phi':S^N \to S^N$ be a bijection satisfying
\begin{enumerate}
\item $\phi'(\xi_1,\ldots, \xi_N) = (\eta_1,\ldots, \eta_N)$;
\item if $(s_1,\ldots, s_N) \in S^N$ satisfies $s_{i+1} \ne s_i^{-1}$ for any $i$ and $\phi'(s_1,\ldots,s_N) = (t_1,\ldots, t_N)$ then $t_{i+1}\ne t_i^{-1}$ for any $i$.
\end{enumerate}
Define $\phi \in \Inn(\cR_0)$ by $\phi(s_1,s_2,\ldots) = (t_1,\ldots, t_N, s_{N+1}, \ldots)$ where $\phi'(s_1,\ldots,s_N)=(t_1,\ldots,t_N)$. This is an inner automorphism of finite order and clearly $\phi(\xi)=\eta$.
\end{proof}

\begin{proof}[Proof of Theorem \ref{thm:extreme1}]
It is immediate from the definition of $\cB$ that $\cB$ is anchored and satisfies the extreme Besicovich property. To see that it is extremely asymptotically invariant, let $\Phi \subset \Inn(\cR_0)$ be the set of finite order inner automorphisms. This is a countable generating set by the previous lemma. Clearly, each $\cB_n$ is preserved under all automorphisms of order $<n$. 
\end{proof}

\subsection{A horospherical ratio pointwise ergodic theorem}\label{sec:horoergodic}

Let $(X,\lambda)$ be a standard $\sigma$-finite measure space on which $\FF$ acts non-singularly. Let $\alpha:\cR_0 \to \FF$ be the cocycle $\alpha(\xi,\eta)=g \in \FF$ where $g\eta = \xi$ (e.g., $g = (\xi_1\cdots \xi_N)(\eta_1\cdots \eta_N)^{-1}$ if $\xi_n=\eta_n$ for all $n\ge N$). Define the equivalence relation $\cR_0^\alpha=\cR_0^\alpha(X)$ on $\partial \FF \times X$ and the sequence of subset functions $\cB^\alpha=\{\cB^\alpha_n\}_{n=1}^\infty$ as in \S \ref{sec:cocycle}. By Theorem \ref{thm:extreme1} and Lemma \ref{lem:cocycle}, it follows that $\cB^\alpha$ is anchored, extremely asymptotically invariant and satisfies the extreme Besicovich property. So Theorem \ref{thm:pointwise} implies:

\begin{thm}\label{thm:pointwise2}
$\cB^\alpha$ is a ratio ergodic sequence in $L^1$. I.e., for every $u,v\in L^1( \partial \FF \times X,\nu\times \lambda)$ with $v > 0$, $\RATIO^{\cB^\alpha}_{n}[u,v]$ converges pointwise a.e. to $\EE_{(\nu\times\lambda)_v}\left[\frac{u}{v}|\cR^\alpha_0\right]$.
 \end{thm}
 
Theorem \ref{thm:main-0} is implied by Theorem \ref{thm:pointwise2} above because $\cB^\alpha(x,\xi)=\{g(x,\xi):~g^{-1}\in H_\xi, |g|\le 2n\}$ implies
 $$\RATIO^{\cB^\alpha}_{n}[u,v]= \frac{ \sum_{g^{-1} \in H_\xi, |g|\le 2n} \hT^gu(x,\xi)}{  \sum_{g^{-1} \in H_\xi, |g|\le 2n} \hT^gv(x,\xi)}.$$

 \section{Ergodic components of the horospherical relation : some examples}\label{sec:ergodic}

Observe that the cocycle $\alpha: \cR_0 \to \FF$ defined above takes values in $\FF^2:=$ the index $2$ subgroup of $\FF$ consisting of all words of even reduced length. Let $\FF \cc (X,\lambda)$ be a non-singular action on a standard $\sigma$-finite measure space. Theorem \ref{thm:pointwise2} raises the following two natural problems. 
 \begin{enumerate}
 \item If $\FF \cc (X,\lambda)$ is ergodic then is the diagonal action $\FF \cc (X\times \partial \FF, \lambda \times \nu)$ ergodic?
 
 \item If the diagonal action $\FF^2 \cc (X\times \partial \FF, \lambda \times \nu)$ is ergodic then is the equivalence relation $\cR_0^\alpha(X)$ ergodic?
 \end{enumerate}

In [BN1] it was shown that, if $(X,\lambda)$ is a probability space and the action preserves the measure then the answer to both questions is `yes'. We show here that both questions have negative answers in general.

\begin{example}
This is a counterexample to (1). Let $X=(\partial \FF\times \partial \FF)\setminus \Delta$ be the set of pairs of distinct points in the boundary $\partial \FF$, which can also be identified with set of bi-infinite geodesics in the Cayley tree.  
As is well known, there exists a Radon measure $\lambda$ on $X$, which is absolutely continuous to the product measure $\nu\times \nu$ and invariant under the action of $\FF$ on $X$ (given by the diagonal action on the product). Moreover $\FF \cc (X,\lambda)$ is ergodic by Kaimanovich's double ergodicity Theorem \cite{Ka03} (which can easily be proven directly in this special case).

Consider now the space $\partial \FF \times X$ with the product measure $\nu\times \lambda$, and the invariant conull set consisting of triples of distinct points. 
It is well-known that the action of $\FF$ on the latter set is in fact proper (as an action on a metric space) and in particular admits a global Borel section. Since the equivalence relation $\cR$ is a sub relation 
of the relation determined by the $\FF$-orbits, it follows that  $\cR$ admits many $\cR$-invariant sets of intermediate measure and is certainly not ergodic. In fact, the $\FF$-action on the product $\partial \FF \times X$ is measure-theoretically smooth. 
\end{example}

\begin{example}
The following provides further counterexamples to (1). Let $G=SL_2(\RR)$, and $X=G/H$, with $H$ a unimodular  connected subgroup, and $\lambda$ taken as the $G$-invariant measure on $G/H$. Let $\FF$ act via a dense homomorphism $\pi : \FF\to G$. 

We want to consider the question of ergodicity of $\FF$ on the product $\partial \FF \times X$, and for that let us view the homomorphism $\pi$  as a cocycle 
$\alpha :  \FF \times \partial\FF \to G$ via $\alpha(g,b)=\pi(g)$. The Mackey range $Y$ of the cocycle is the space of ergodic components of the  induced action of $\FF$ on $\partial \FF \times G$ (by $f(b,g)=(fb, \alpha(f,b)g)$). The group $G$ acts on $\partial \FF \times G$ via $g(f,h)=(f,hg^{-1})$. Because this action commutes with the action of $\FF$, it descends to an action on $Y$. Since the action of $\FF$ on $\partial \FF$ is amenable and ergodic, $G$ acts amenably and ergodically on $Y$. By \cite{Z3}, any ergodic amenable action of $G$ factors over a transitive amenable action $G/L$, with $L$ being a closed amenable subgroup.  

 Since $L$ is amenable, it leaves a probability measure on the boundary of hyperbolic space invariant, and therefore it is either conjugate to a subgroup of the maximal compact subgroup $K$, or conjugate to a subgroup of the minimal parabolic subgroup $P$, or to a subgroup of the group $N_K(A)$, the normalizer in $K$ of the diagonal group $A$. By passing to a further factor  if necessary, we can therefore assume either $L=K$, $L=P$ or $L=N_K(A)$. 

Since $G/L$ is a factor of the Mackey range of the cocycle $\alpha$, it follows that $\alpha$ is cohomologous to a cocycle taking values in $L$, and this is equivalent to $\partial \FF$ 
admitting an $\alpha$-invariant map $\psi$ to $G/L$. Since the cocycle is defined by $\alpha(g,b)=\pi(g)$, this 
amounts to a measurable non-singular map $\psi$, equivariant with respect to the actions of $\FF$ on both sides. 
%Note that $\partial \FF$ supports an ergodic probability measure stationary under $\FF$, and therefore so does any equivariant factor. However $SL_2(\RR)/L$ does not support  such a measure unless $L=P$. It follows that the case where $L$ is a subgroup of $K$ or $N_K(A)$ cannot arise, 

Thus we have 
a non-singular factor map $\psi : \partial \FF\to   G/L$.  Now in general (see \cite{Z2}), for any $G$-space $X$, $\FF$ is ergodic on $\partial \FF\times X$ (acting via $\pi$) if and only if $G$ is ergodic on $Y\times X$.  Choosing $X=G/H$,  by Moore's ergodicity theorem (see e.g. \cite{Z1}),  $G$ is ergodic on $G/L\times G/H$ if and only if $H$ is ergodic on $G/L$. Since the Mackey range $Y$ factors over $G/L$,  it follows of course that  $G$ is not ergodic on $Y\times G/H$ if $H$ is not ergodic on $G/L$.  Choosing $H$ to be the trivial group $\set{e}$, the maximal compact subgroup $K$,  or  the diagonal subgroup $A$, we obtain that $H$ is not ergodic on $G/L$, for any of the admissible  choices of $L$, namely $L=K$, $L=P$ or $L=N_K(A)$.  

Thus taking $X$ to be the group itself ($X=SL_2(\RR)$), or the hyperbolic plane ($X=\HH^2=G/K$) or the de-Sitter space ($X=G/A$), we conclude that the $\FF$-action on $\partial \FF\times X$ is not ergodic. 
\end{example}

\begin{example}
To obtain a counterexample to question (2), we consider the action of $\FF$ on $(\partial \FF \times \partial \FF, \nu \times \nu)$. This action is ergodic (e.g. by Kaimanovich's double ergodicity theorem, which can easily be proven directly in this special case). The induced action of $\FF^2$ on  $(\partial \FF \times \partial \FF, \nu \times \nu)$ is also ergodic. However, the equivalence relation $\cR_0^\alpha(\partial \FF)$ is not ergodic. To see this, note that a point $(b,c) \in \partial \FF \times \partial \FF$ with $b\ne c$ determines a geodesic $\gamma:\ZZ \to \FF$ from $b$ to $c$. Let $H_c$ be the horosphere centered at $c$ that passes through the identity element. Without loss of generality we may assume that $\gamma(0)$ is contained in $H_c$. Define $J_{(b,c)}(n) = \gamma(n)^{-1}\gamma(n+1)$. We claim that the map $(b,c) \in \partial \FF \times \partial \FF \mapsto J_{(b,c)}$ is invariant under the relation $\cR_0^\alpha(\partial\FF)$. Because the map $(b,c) \mapsto J_{(b,c)}$ is not constant a.e. this proves that $\cR_0^\alpha(\partial\FF)$ is not ergodic. 

Any element in the $\cR_0^\alpha(\partial\FF)$-equivalence class of $(b,c)$ has the form $(g^{-1}b,g^{-1}c)$ where $g\in H_c$. Notice that $g^{-1}H_c = H_{g^{-1}c}$. Therefore, $g^{-1}\gamma$ is the geodesic from $g^{-1}b$ to $g^{-1}c$ and $g^{-1}\gamma(0)$ is contained in $H_{g^{-1}c}$. So if $J'=J_{(g^{-1}b,g^{-1}c)}$ then 
$$J'(n) = (g^{-1}\gamma(n))^{-1} g^{-1}\gamma(n+1) = \gamma(n)^{-1}\gamma(n+1) = J_{(b,c)}(n).$$
This proves the claim. 
\end{example}

It would be useful to establish general conditions on an ergodic non-singular action $\FF \cc (X,\lambda)$ which imply positive answers to questions (1) and (2) above.  As am example, note that even the following is not known. 
\begin{question}
Suppose $\FF \to \text{Isom}(\EE^n)$ is a dense embedding into the isometry group of Euclidean $n$-space. Through this embedding, $\FF$ acts by isometries on Euclidean $n$-space. Is the induced action $\FF \cc (\EE^n \times \partial \FF, \lambda^n \times \nu)$ ergodic? Here $\lambda^n$ is the Lebesgue measure on $\EE^n$. If so, then is the equivalence relation $\cR_0^\alpha({\EE^n})$ ergodic?
\end{question}

\end{document}